\documentclass[11pt]{article}
\usepackage{amsmath,amssymb,amsthm}
\usepackage{epsfig}
\usepackage{color}
\usepackage{enumerate}
\usepackage{algorithm}
\usepackage{algorithmic}
\usepackage{hyperref}

\hypersetup{colorlinks=true}

\hypersetup{colorlinks=true, linkcolor=blue, citecolor=blue,urlcolor=blue}

\title{Packing and domination parameters in digraphs}

\author{Doost Ali Mojdeh$^1$, Babak Samadi$^2$ and Ismael G. Yero$^3$\\[0.5cm]
$^{1,2}$Department of Mathematics, University of Mazandaran, Babolsar, Iran\\
{\it $^1$damojdeh@umz.ac.ir,} {\it $^2$samadibabak62@gmail.com}\\[0.2cm]
$^3$Departamento de Matem\'{a}ticas, Universidad de C\'{a}diz, Algeciras, Spain\\
{\it $^3$ismael.gonzalez@uca.es}}
\date{}

 \newtheorem{theorem}{Theorem}[section]

\newtheorem{lemma}[theorem]{Lemma}

\newtheorem{p}{Problem}

\theoremstyle{definition}
\newtheorem{definition}[theorem]{Definition}
\newtheorem{rem}[theorem]{Remark}

\begin{document}

\maketitle
\begin{abstract}
Given a digraph $D=(V,A)$, a set $B\subset V$ is a packing set in $D$ if there are no arcs joining vertices of $B$ and for any two vertices $x,y\in B$ the sets of in-neighbors of $x$ and $y$ are disjoint. The set $S$ is a dominating set (an open dominating set) in $D$ if every vertex not in $S$ (in $V$) has an in-neighbor in $S$. Moreover, a dominating set $S$ is called a total dominating set if the subgraph induced by $S$ has no isolated vertices. The packing sets of maximum cardinality and the (total, open) dominating sets of minimum cardinality in digraphs are studied in this article. We prove that the two optimal sets concerning packing and domination achieve the same value for directed trees, and give some applications of it. We also show analogous equalities for all connected contrafunctional digraphs, and characterize all such digraphs $D$ for which such equalities are satisfied.  Moreover, sharp bounds on the maximum and the minimum cardinalities of packing and dominating sets, respectively, are given for digraphs. Finally, we present solutions for two open problems, concerning total and open dominating sets of minimum cardinality, pointed out in [Australas. J. Combin. 39 (2007), 283--292].\vspace{1mm}

\vspace{.5mm}
\end{abstract}

%%%%%%%%%%%%%%%%%%%%%%%%%%%%%%%%%%%%%%%%%%%

\section{Introduction}

Aspects concerning domination (and packings) in graphs have attracted the attention of a high number of researchers in the last few decades. The topic has found a number of applications to several real life problems and there are numerous problems on domination which remains open. For more information on domination topics we suggest the books \cite{hhs2,HeYe} and references cited therein. Domination topics in digraphs are less common, although a significant increment of them can be noticed in the last five years and a number of open problems is being raised up. It is then, a goal of this work, to give several good results concerning relationships between some different styles of domination parameters for digraphs, and meanwhile, settle two open problems which are already known in this topic.

On the other hand, we may remark that our study also contributes to decreasing the not balanced situation existent in the literature between graphs and digraphs. That is, graphs and (directed graphs) digraphs are mathematical structures which naturally appear in several models of real life problems, and actually digraphs are very frequently more realistic than graphs while modeling a situation. However, the study of both structures is not in correspondence with this fact. The theory of graphs is significantly more developed than the theory of digraphs. For instance, if we simply make a query at the MathSciNet database with the word ``graph'' we get 73969 articles (32355 in the last ten years), while a similar query with the word ``digraph'' gives an answer with only 3246 results (1288 in the last ten years). That is clearly not fair. A similar situation occurs if we join such words with an extra term. For instance, the words ``domination number'', ``dominating set'' and ``packing'' (subjects of this work) produce similar results. In this sense, throughout this exposition we significantly contribute to decreasing such not balanced relationship between graphs and digraphs for the specific case of (total, open) dominating sets and packing sets.

Throughout this paper, we consider $D=(V(D),A(D))$ as a finite digraph with vertex set $V=V(D)$ and arc set $A=A(D)$ with neither loops nor multiple arcs (although pairs of opposite arcs are allowed). Also, $G=(V(G),E(G))$ stands for a simple finite graph with vertex set $V(G)$ and edge set $E(G)$. We use \cite{bg} and \cite{we} as references for some very basic terminology and notation in digraphs and graphs, respectively, which are not explicitly defined here.

For any two vertices $u,v\in V(D)$, we write $(u,v)$ as the \emph{arc} with direction from $u$ to $v$, and say $u$ is \emph{adjacent to} $v$, or $v$ is \emph{adjacent from} $u$. Given a subset $S$ of vertices of a digraph $D$ and a vertex $v\in V(D)$, the {\em in-neighborhood} of $v$ from $S$ ({\em out-neighborhood} of $v$ to $S$) is $N_S^{-}(v)=\{u\in S\mid(u,v)\in A(D)\}$ ($N_S^{+}(v)=\{u\in S\mid(v,u)\in A(D)\}$). The \emph{in-degree} of $v$ from $S$ is $deg_S^-(v)=|N_S^{-}(v)|$ and the \emph{out-degree} of $v$ to $S$ is $deg_S^+(v)=|N_S^{+}(v)|$. Moreover, $N_S^{-}[v]=N_{S}^{-}(v)\cup\{v\}$ is the {\em closed in-neighborhood} of $v$ from $S$ ($N_S^{+}[v]=N_{S}^{+}(v)\cup\{v\}$ is the {\em closed out-neighborhood} of $v$ to $S$). If particularly, $S=V(D)$, then we simply say (open or closed) (in or out)-neighborhood and (in or out)-degree, and write $N_D^{+}(v)$, $N_D^{-}(v)$, $N_D^{+}[v]$, $N_D^{-}[v]$, $deg_D^+(v)$ and $deg_D^-(v)$ (or $N^{+}(v)$, $N^{-}(v)$, $N^{+}[v]$, $N^{-}[v]$, $deg^+(v)$ and $deg^-(v)$ if there is no ambiguity with respect to the digraph $D$), instead of $N_{V(D)}^{+}(v)$, $N_{V(D)}^{-}(v)$, $N_{V(D)}^{+}[v]$, $N_{V(D)}^{-}[v]$, $deg_{V(D)}^+(v)$ and $deg_{V(D)}^-(v)$, respectively. We similarly proceed with any other notation which uses such style of subscripts. Let $S\subseteq V(D)$ and $u\in S$. A vertex $v$ in $V(D)$ is called a {\em private out-neighbor} ({\em private in-neighbor}) of $u$ with respect to $S$ if $N^{-}[v]\cap S=\{u\}$ ($N^{+}[v]\cap S=\{u\}$). The set of all private out-neighbors (private in-neighbors) of $u$ with respect to $S$ is denoted by $pn^{+}(u,S)$ ($pn^{-}(u,S)$). Given two sets $A$ and $B$ of vertices of $D$, by $(A,B)_D$ we mean the sets of arcs of $D$ going from $A$ to $B$, that is, $(A,B)_D=\{(a,b)\in A(D)\mid a\in A,b\in B\}$.

A digraph $D$ is {\em connected} if its underlying graph is connected. A {\em rooted tree} is a connected digraph with a vertex of in-degree $0$, called the {\em root}, such that every vertex different from the root has in-degree $1$. In a rooted tree, the vertex of out-degree $0$ is called a {\em leaf} and its in-neighbor is a {\em support vertex}.
A {\em binary tree} is a rooted tree in which the number of out-neighbors of each vertex in zero or two. The {\em height} $h(T)$ of a rooted tree $T$ is $h(T)=\mbox{max}\{d_{T}(r,v)\mid r\ \mbox{is the root and}\ v\in V(T)\}$. A {\em directed star} $S_{n}$ on $n$ vertices is a rooted tree of order $n$ with $h(S_{n})=1$. A digraph $D$ is {\em contrafunctional} if every vertex of $D$ has in-degree one.

A {\em $k$-sun} on $2k$ vertices is a construction starting with a Hamiltonian graph $G$ of order $k$, with Hamilton cycle $v_{1},\dots,v_{k}$, next $k$ new vertices $u_{1},\dots,u_{k}$ are added so that each $u_{i}v_{i},u_{i}v_{i+1}\in E(G)$ (mod $k$). A vertex $v$ of $G$ is {\em simplicial} if $N[v]$ induces a clique. A {\em simplicial elimination ordering} is an ordering $v_{n},\dots,v_{1}$ for deletion of vertices so that each vertex $v_{i}$ is a simplicial vertex of the remaining graph induced by $\{v_{1},\dots,v_{i}\}$. A graph $G$ is {\em chordal} if it has no induced cycle with four vertices or more, and $G$ is {\em strongly chordal} if it is chordal and contains no $k$-suns as induced subgraphs. We have the following classic result of Dirac (\cite{d}) concerning chordal graphs.
\begin{theorem}\emph{(Dirac (1961))}\label{Dirac}
A simple graph is chordal if and only if it has a simplicial elimination ordering of vertices.
\end{theorem}

Our work is organized as follows. The next subsection is dedicated to describe some terminology and notation which we shall use throughout our exposition. Section \ref{sect-trees} is centered in the study of directed trees. For instance, we prove that the two optimal sets concerning packing and domination achieve the same value for directed trees, and give some applications of it. Section \ref{sect-gen-contrafunc} gives more general results and specify some other other ones for the case of connected contrafunctional digraphs. That is, we show a bound for the packing number of digraphs, and also prove some analogous equalities, as those ones in trees, for all connected contrafunctional digraphs. We moreover, characterize all such digraphs $D$ for which such equalities are satisfied.  Finally, in Section \ref{sect-total-open} we present some sharp bounds on the maximum and the minimum cardinalities of packing and dominating sets, respectively, which are satisfied by digraphs. We also show here the solutions for two open problems, concerning total and open dominating sets of minimum cardinality, pointed out in [Australas. J. Combin. 39 (2007), 283--292].

\subsection{Terminology on packing and (total, open) domination}

\ \ \ Given a graph $G=(V,E)$, a set $S\subseteq V(G)$ is a {\em dominating set} (a {\em total dominating set}) in $G$ if each vertex in $V(G)\setminus S$ (in $V(G)$) is adjacent to at least one vertex in $S$. The {\em domination number} $\gamma(G)$ ({\em total domination number} $\gamma_{t}(G)$) is the minimum cardinality of a dominating set (a total dominating set) in $G$. A subset $B\subseteq V(G)$ is a {\em packing} set (or a {\em $2$-packing set} as also appeared in the literature) in $G$ if for every distinct vertices $u,v\in B$, $N[u]\cap N[v]=\phi$ (notice that $N[x]$ is the closed neighborhood of $x$ while we do not consider directions of the edges). The {\em packing number} (or {\em $2$-packing number}) $\rho(G)$ is the maximum cardinality of any packing set in $G$. Clearly, $B\subseteq V(G)$ is a packing set in $G$ if and only if $|N[v]\cap B|\leq 1$, for all $v\in V(G)$.

The concepts concerning domination in directed graphs were introduced by Fu \cite{fu} as follows. A subset $S$ of the vertices of a digraph $D$ is called a \emph{dominating set} if every vertex in $V(D)\setminus S$ is adjacent from a vertex in $S$. Now, if one thinks into consider a total domination version for digraphs, it is possible to find two different versions of it in the literature. In one side, from \cite{ajv} we have the next definition. A dominating set $S$ in $D$ is called a \emph{total dominating set} if $D\langle S\rangle$ has no isolated vertices. On a second side, if we read the article \cite{os} for instance, we get the following different definition. A \emph{total dominating set} of a digraph $D$ is a vertex subset $S$ such that any vertex of $D$ is adjacent from a vertex of $S$. Clearly, both definitions are different, and if we longer observe the literature, we will notice that the latter structure coincides with that one called \emph{open dominating sets} (for digraphs) in \cite{ajv}. Moreover, a deeper search in the literature will lead to the fact that the most common definition for total dominating sets in digraphs is this one given in \cite{ajv} (see for instance \cite{h} and references cited therein). Thus, from now on, we assume in this work the definition of total dominating sets in digraphs as given in \cite{ajv}, although we consider that the definition given in \cite{os} as more natural and more according to its non directed version. The domination number $\gamma(D)$, the total domination number $\gamma_{t}(D)$ and the open domination number $\gamma_{o}(D)$ are defined in a natural way, similarly as they are in graphs. From now on, given any parameter $P$ in a graph $G$ (or a digraph $D$), a set of vertices of cardinality $P(G)$ (or $P(D)$) is called a $P(G)$-set (or $P(D)$-set).

For the sake of more exploration into the concept of domination in digraphs, we investigate the concept of packing parameter in digraphs. Volkmann \cite{v} introduced the packing number in digraphs as follows, although such definition has passed unnoticed for the research community since a unique result concerning it was given in such work. It is also now our goal to make some justice to such parameter and properly begin the study of its mathematical properties. A set $B\subseteq V(D)$ is a \emph{packing set} in a digraph $D$ if $N^{-}[u]\cap N^{-}[v]=\phi$ for any two distinct vertices $u,v\in B$. The maximum cardinality of a packing is the \emph{packing number} of $D$, denoted by $\rho(D)$. In what follows, we would prefer to present an equivalent definition of it.

\begin{definition}
The set $B\subseteq D$ is a {\em packing} in $D$ if $|N^{+}[v]\cap B|\leq 1$ for all $v\in V(D)$, and the {\em  packing number} $\rho(D)$ is the largest number of vertices in a packing set of $D$.
\end{definition}

We prove that $\rho(T)=\gamma(T)$, for all directed trees $T$. Using this fact, we show that $\lceil(n-\ell+s)/2\rceil$ is a sharp upper bound on the domination number of a rooted tree $T$ of order $n$ with $\ell$ leaves and $s$ support vertices. Concerning all connected contrafunctional digraph $D$ we prove that $\gamma(D)\in\{\rho(D),\rho(D)+1\}$ and characterize all such digraphs $D$ for which $\gamma(D)=\rho(D)$  and $\gamma(D)=\rho(D)+1$. We give the characterization of all digraphs $D$ of order $n$ with maximum out-degree $\Delta^+$ for which $\gamma_{t}(D)=2n/(2\Delta^++1)$ and  $\gamma_{o}(D)=n/\Delta^+$, hence solving two open problems pointed out in \cite{ajv}.

%%%%%%%%%%%%%%%%%%%%%%%%%%%%%%%%%%%%%%%%%%%%%%%%%%%%%%%%%%%%%%%%%%%%%%%%%%%%%%%

\section{Directed trees}\label{sect-trees}

\ \ \ \ In this section, we study some relationships between packing and domination numbers in digraphs with emphasis on directed trees. We might remark that a directed tree is an orientation of a tree, which in other words means that it cannot have opposite arcs. We first exhibit the following useful construction.

\begin{rem}\label{SM}
Consider a digraph $D=(V(D),A(D))$. We construct a digraph $G'_{D}$ corresponding to $D$, as follows. For each vertex $v$ of $D$ consider two vertices $v$ and $v'$ and an arc $(v,v')$ for $G'_{D}$. Moreover, if there is an arc $(v_i,v_j)$ in $D$, then we add two arcs $(v_i,v_j)$ and $(v_i,v_j')$ in $G'_{D}$ (note that every vertex $v$ of $D$ is corresponding to the directed path $v,v'$ in $G'_{D}$). Now, we define $G_{D}$ as the underlying graph of $G'_{D}$.
\end{rem}

By using the remark above, we present the following lemma which might be useful in its own. We make use of the notation used in the definition of $G_D$.

\begin{lemma}\label{A}
For any digraph $D$ of order $n$,
\begin{enumerate}[{\rm (i)}]
  \item $\gamma(G_{D})=\gamma(D)$,
  \item $\rho(G_{D})=\rho(D)$.
\end{enumerate}
\end{lemma}

\begin{proof}
(i) Let $S=\{v_{1},\dots,v_{|S|}\}$ be a $\gamma(D)$-set and consider the set $S'$ in $G_D$ formed by the corresponding vertices of $S$. Let $x\in V(G_D)-S'$. Clearly, if $x=v_i'$ whether $v_i\in S$, then $x$ is dominated by $v_{i}\in S'$. Suppose now that $x\ne v'_i$ for any $i\in\{1,\dots,|S|\}$. Since there exists an arc $(v_{i},y)$ for every $y\notin S$, there exist edges between $v_i$ and $y$, and between $v_i$ and $y'$. In this sense, there must exist a vertex $v_j\in S'$ for some $j\in\{1,\dots,|S|\}$ such that $x$ is dominated by $v_j$. Thus, $S'$ is a dominating set in $G_D$ and so, $\gamma(G_{D})\le \gamma(D)$.

Suppose now that $P$ is a $\gamma(G_{D})$-set. Taking into account that any vertex $v$ of $G_{D}$ dominates any vertex dominated by $v'$, we may assume that $P$ does not contain vertices of $v'$ style. If this is the case, we simply replace each $v'$ with $v$ or remove $v'$ if $v,v'\in P$ (which is indeed not possible since $P$ is a $\gamma(G_{D})$-set). Hence, let $P=\{v_{1},\dots,v_{|P|}\}$ and let $P'$ be the set of vertices of $D$ corresponding to the vertices in $P$. Let $x\in V(D)\setminus P'$. Note that this means the corresponding vertices $x,x'$ of $G_D$ are not in $P$. So, there are two edges $v_{j}x$ and $v_{j}x'$ in $G_{D}$ for some $j\in\{1,\dots,|P|\}$, implying that $(v_{j},x)\in A(D)$. Therefore, $P'$ is a dominating set in $D$ and consequently, $\gamma(D)\leq|P|=\gamma(G_{D})$.

(ii) Let $B=\{v_{1},\dots,v_{|B|}\}$ be a $\rho(D)$-set. We claim that $B'=\{v'_{1},\dots,v'_{|B|}\}$ is a packing in $G_{D}$. It is easy to see that $B'$ is independent. Now, suppose that there is a vertex $x\in V(G_{D})$ having two distinct neighbors $v'_{i},v'_{j}\in B'$ (notice that such $x$ must be a vertex with $v$ style, not $v'$ style). Hence, $\{v_{i},v_{j}\}\subseteq N^+[x]\cap B$, a contradiction. Therefore, $B'$ is a packing in $G_{D}$ and so, $\rho(G_{D})\geq|B|=\rho(D)$.

On the other hand, suppose that $Q$ is a $\rho(G_{D})$-set. First, note that for every vertex $v\in Q$, $(Q\setminus\{v\})\cup\{v'\}$ is a $\rho(G_{D})$-set as well. So, we may assume that $Q$ does not contain vertices of $v$ style, since it cannot also happen that $v,v'\in Q$. Consider now $R=\{v\in V(D)\ |\ v'\in Q\}$. If there is an arc $(x,y)$ in $D$ for some $x,y\in R$, then $\{x',y'\}\subseteq N[x]\cap Q$, a contradiction. So, $R$ is independent. If there is vertex $z\in V(D)\setminus R$ having two out-neighbors $v_i,v_j\in R$, then $\{v'_i,v'_j\}\subseteq N[z]\cap Q$, which is again a contradiction. Therefore, $R$ is a packing in $D$ and so, $\rho(G_{D})=|R|\leq \rho(D)$. This completes the proof.
\end{proof}

It is well known that the inequality $\rho(G)\leq \gamma(G)$ holds for any graph $G$ (\cite{hhs2}). This fact together with Lemma \ref{A} (and using the construction $G_D$ made in Remark \ref{SM}) lead to the following immediate consequence.

\begin{rem}\label{EQ2}
For any digraph $D$, $\rho(D)=\rho(G_{D})\leq \gamma(G_{D})=\gamma(D).$
\end{rem}

We note that the difference between these two digraph parameters can be arbitrary large. For instance, as a well-known result, there exist tournaments with arbitrary large domination number (see \cite{e}), while $\rho(D)=1$ for each tournament $D$.

We now center our attention in directed trees. In connection with this (for the non-directed case), Meir and Moon \cite{mm} showed that $\rho(T)=\gamma(T)$, for all trees $T$ and, in a more general case, the following result due to Farber \cite{f1} (see also \cite{cn}) is known.

\begin{lemma}\label{B}
If a graph $G$ is strongly chordal, then $\rho(G)=\gamma(G)$.
\end{lemma}

We are now aimed to present the following theorem, which can be considered as a directed version of the classic result of Meir and Moon \cite{mm}.

\begin{theorem}\label{T2}
If $T$ is a directed tree, then $\rho(T)=\gamma(T)$.
\end{theorem}

\begin{proof}
Let $T$ be a directed tree with $V(T)=\{v_{1},\dots,v_{n}\}$. In order to complete our proof we use the structure $G_T$ defined in Remark \ref{SM}. We first show that $G_{T}$ has a simplicial elimination ordering (or equivalently $G_T$ is a chordal graph according to Theorem \ref{Dirac}). We employ induction on the order $n$ of $T$. It is clearly obvious for $n=1$. We suppose now that it is true for any directed tree of order $n-1$ and shall consider a directed tree of order $n$. Let $u$ be a vertex of $T$ of degree one. Applying the inductive hypothesis to $T-u$, we obtain a simplicial elimination ordering $u_{2n-2},\dots,u_{1}$ for $G_{T-u}$. It is then readily seen that either $\mbox{deg}^{-}(u)=1$ or $\mbox{deg}^{+}(u)=1$ would occur, we always obtain a simplicial elimination ordering $u',u,u_{2n-2},\dots,u_{1}$ for $G_{T}$.

We next claim that $G_{T}$ has no cycles of length $l\geq 6$. For the contrary, suppose $C$ is a cycle of length at least six in $G_{T}$. If $V(C)\subseteq\{v_{1},\dots,v_{n}\}$, then the vertices in $V(C)$ are on a cycle in the undirected underlying tree $T'$ of $T$, which is a contradiction. So, $V(C)$ must contain vertices in the style $v'_{i}$, for some $1\leq i\leq n$. Moreover, since the vertices $v'_{1},\dots,v'_{n}$ are independent, there cannot be two consecutive vertices in $C$ of the $v'_i$ style. Let $P:v_{x},v'_{a},v_{y}$ be a path on $C$. Clearly, $|V(C)\setminus\{v_{x},v'_{a},v_{y}\}|\geq 3$. We consider the following situations.

\noindent
\textbf{Case 1:} $|(V(C)\setminus\{v_{x},v'_{a},v_{y}\})\cap\{v_{1},\dots,v_{n}\}|=1$. Assume  $v_{b}$ is the only member of $(V(C)\setminus\{v_{x},v'_{a},v_{y}\})\cap\{v_{1},\dots,v_{n}\}$. Hence, by taking into account the fact that the vertices $v'_{i}$ are independent, it must happen that the adjacency in the vertices on $C$ follows the order $v_{x},v'_{a},v_{y},v'_{c},v_{b},v'_{d},v_x$ for some different vertices $v'_c,v'_d$. By observing the adjacency conditions of $G_T$, we can deduce there exists a cycle in $T'$ with vertices in $\{v_{x},v_{a},v_{y},v_{c},v_{b},v_{d},v_x\}$ which is not possible.

\noindent
\textbf{Case 2:} $|(V(C)\setminus\{v_{x},v'_{a},v_{y}\})\cap\{v_{1},\dots,v_{n}\}|\geq 2$. Let $v_{f}$ and $v_{g}$ be two distinct vertices in $(V(C)\setminus\{v_{x},v'_{a},v_{y}\})\cap\{v_{1},\dots,v_{n}\}$. It is not difficult to see that there exists a $v_{x},v_{y}$-path $P'$ in $T'$ in which $v_{f}$ and $v_{g}$ (at least one of them is different from $v_{a}$) appear. So, there are two different $v_{x},v_{y}$-paths in $T'$: $P'$ and $P''=v_{x},v_{a},v_{y}$, which is a contradiction.

As a consequence of the two cases above, we obtain that $G_{T}$ has no cycle of length at least six. Now, if $G_{T}$ has a $k$-sun $H$ as an induced subgraph, then the vertices of $H$ would be on a cycle of length $l=2k\geq 6$, which is a contradiction. Thus, $G_{T}$ is $k$-sun free, for each $k\geq 3$ and, consequently with the reasoning till this moment, we claim that $G_{T}$ is a strongly chordal graph. Therefore, by Lemma \ref{A} and Lemma \ref{B}, we have $\rho(T)=\rho(G_{T})=\gamma(G_{T})=\gamma(T),$ which completes the proof.
\end{proof}

We turn now our attention to rooted trees and let $v$ be the root of a rooted tree $T$. We construct a set $B\subseteq V(T)$ by the following process. Suppose that $T=T_{1}$ and select a leaf $v_{1}$ with maximum distance from $v$ as a member of $B$. Let $u_{1}$ be the support vertex adjacent to $v_{1}$ and let $T_{2}=T_{1}-N^+[u_{1}]$. Iterate this process, in which we always chose as a member of $B$, a leaf $v_i$ of the rooted tree $T_i=T_{i-1}-N^+[u_{i-1}]$ at a maximum distance from $v$. We end the process whether we have removed all vertices or get an isolated vertex, in which case we put such isolated vertex into $B$. Let $B=\{v_{1},\dots, v_{|B|}\}$ (notice that $v_{|B|}$ is either $v$ or one of its out-neighbors). From the above procedure, it is readily seen that $B$ is a packing in $T$. Thus, $|B|\leq \rho(T)$.

On the other hand, we consider the following situations.
\begin{itemize}
  \item If the process described above ends with all the vertices removed, then we note that $\mathbb{P}_1=\{N_{T_1}^+[u_{1}], \dots, N_{T_{|B|}}^+[u_{|B|}]\}$ is a partition of $V(T)$ where $u_1,\dots,u_{|B|}$ are the support vertices of $v_1,\dots,v_{|B|}$ in $T_1,\dots, T_{|B|}$, respectively.
  \item If the process described above ends with an isolated vertex as the rooted tree $T_{|B|}$, then we again observe that $\mathbb{P}_2=\{N_{T_1}^+[u_{1}], \dots, N_{T_{|B|-1}}^+[u_{|B|-1}],\{v\}\}$ is a partition of $V(T)$ where $u_1,\dots,u_{|B|-1}$ are the support vertices of $v_1,\dots,v_{|B|-1}$ in $T_1,\dots, T_{|B|-1}$, respectively.
\end{itemize}
Since both partitions $\mathbb{P}_1$ and $\mathbb{P}_2$ have the same cardinality, we may assume $\mathbb{P}=\{P_1,\dots, P_{|B|}\}$ is a partition of $V(T)$ given in one of the above ways, as the situation would correspond. If $Q$ is a $\rho(T)$-set, then it is clearly satisfied that $|Q\cap P_i|\leq 1$, for all $1\leq i\leq |B|$. Therefore, $\rho(T)=|Q|=\sum^{|B|}_{i=1}|Q\cap P_i|\leq |B|,$ which leads to $\rho(T)=|B|$.

In the process showen above we recursively eliminate a ``special directed star'' in each step, and obtain a sequence $(T=T_{1},\dots,T_{p-1},T_{p})$, in which $T_{1},\dots,T_{p-1}$ are rooted trees and $T_{p}$ is empty or just an isolated vertex. From now on we call it a {\em recursive directed star elimination sequence} (briefly RDSES). According to the process, for any rooted tree $T$ we are able to provide an algorithm for finding a $\rho(T)$-set. In other sense, it can be also considered as an algorithm for computing the domination number of $T$, by Theorem \ref{T2}. The algorithm, which is next stated, uses the so-called Breadth-First Search (BFS for short) algorithm (\cite{moore,zuze}) for traversing the vertices of the underling tree of the rooted tree $T$. In the algorithm, given a vertex $x$, by $p(x)$ we mean the parent of $x$, and by $Ch(x)$, the children of $x$ (see \cite{we}).

\begin{algorithm}
\caption{Maximum packing set}\label{algor}
\vspace*{0.2cm}
Input: A rooted tree $T$ of order $n\geq 2$ with root $u$\\
Output: a $\rho(T)$-set
\begin{algorithmic}
\STATE $B=\emptyset$
\STATE order $V(T)$ by BFS algorithm
\STATE $L$ is the list of vertices ordered with respect to the BFS-ordering
\WHILE{$|L|\ge 1$}
\STATE take last vertex $v\in L$
\STATE add $v$ to $B$
\STATE remove $p(v)$ and $Ch(p(v))$ from $L$ (note that $v\in Ch(p(v))$)
\ENDWHILE
\end{algorithmic}
\end{algorithm}

Since it is well known that the BFS algorithm runs in linear time for trees, and according to the previously described process, it is easy to check that the Algorithm \ref{algor} runs in polynomial-time on the order of the rooted tree $T$. On the other hand, note that the problem of finding a minimum dominating set (MDS) in strongly chordal graphs is polynomial-time solvable (see \cite{f}). Since $G_{T}$ is strongly chordal (as shown in the proof of Theorem \ref{T2}), the equality $\gamma(T)=\gamma(G_{T})$ reduces the problem of finding an MDS in directed trees to the problem of finding an MDS in strongly chordal graphs. So, the problem for directed trees is polynomial-time solvable, as well.

In what follows, we bound $\rho(T)$ on a rooted tree $T$ from below and above. They can be considered as bounds on $\gamma(T)$ in view of Theorem \ref{T2}. To this end, we need a result proven by Lee in \cite{l1}, which state that that for any rooted tree of order $n$,
\begin{equation}\label{EQ6}
\gamma(T)\leq \left\lceil\frac{n}{2}\right\rceil.
\end{equation}

\begin{theorem}\label{T1}
Let $T$ be a rooted tree of order $n\geq2$ with $\ell$ leaves and $s$ support vertices. Then,
$$s\le \rho(T)\leq\left\lceil\frac{n-\ell+s}{2}\right\rceil.$$
\end{theorem}

\begin{proof}
Let $B$ be a $\rho(T)$-set in $T$. It can be readily seen that there is one element in $B$ for each support vertex of $T$, which leads to the lower bound.

For any support vertex $u$ of $T$, we let $\ell(u)$ be the number of leaves adjacent from $u$. Let $T'$ be obtained from $T$ by deleting $\ell(u)-1$ leaves adjacent from $u$, for each support vertex $u$ of $T$. Clearly, $n'(T)=n-\ell+s$. Note that at most one leaf belongs to $B$, for each support vertex of $T$. Therefore, $\rho(T)=|B|=\rho(T')$. On the other hand, $\rho(T')=\gamma(T')$ by Theorem \ref{T2}. So, by using (\ref{EQ6}), we deduce
\begin{equation}\label{eq54}
\rho(T)=\gamma(T')\leq\lceil n(T')/2\rceil=\left\lceil(n-\ell+s)/2\right\rceil,
\end{equation}
which gives our desired upper bound.
\end{proof}

We now center our attention into characterizing all rooted trees attaining the bounds in Theorem \ref{T1}. In this sense, let $\Phi$ be the family of all rooted trees $T$ satisfying:
\begin{itemize}
  \item[{\rm(a)}] $V(T)$ has the partition $\mathbb{P}_{1}$ in which every member of it is isomorphic to $S_{2}$, or
  \item[{\rm(b)}] $V(T)$ has the partition $\mathbb{P}_{2}$ in which the directed stars $N_{T_i}^+[u_{i}]$ are isomorphic to $S_{2}$, or
  \item[{\rm(c)}] $V(T)$ has the partition $\mathbb{P}_{2}$ in which one of the directed stars $N_{T_i}^+[u_{i}]$ is isomorphic to $S_{3}$ and the others are isomorphic to $S_{2}$, for $1\leq i\leq|B|$.
\end{itemize}
This family was defined in \cite{ha} by using a different notation. It is also known from \cite{ha} that $\gamma(T)=\lceil n/2\rceil$ if and only if $T\in \Phi$. In our characterization, we use some notation defined in the proof of the upper bound in Theorem \ref{T1}. Namely, the rooted tree $T'$ obtained from a rooted tree $T$ by removing $\ell(u)-1$ leaves adjacent from $u$, for each support vertex $u$ of $T$.

\begin{theorem}\label{th-charact-rho-s}
Let $T$ be a rooted tree of order $n\geq2$ with $\ell$ leaves and $s$ support vertices. Then,
\begin{itemize}
  \item[{\rm (i)}] $\rho(T)=s$ if and only if either $n=\ell+s$, or $n\ne \ell+s$ and every non support vertex and non leaf of $T$ is adjacent from a support vertex of $T$.
  \item[{\rm (ii)}] $\rho(T)=\lceil(n-\ell+s)/2\rceil$ if and only if $T'\in \Phi$ $(T'$ is the rooted tree obtained from $T$ as previously described$)$.
\end{itemize}
\end{theorem}

\begin{proof}
(i) If $n=\ell+s$, then, by using Theorem \ref{T1}, we can easily notice that $\rho(T)=s$. Assume now that $n\ne \ell+s$ and that every non support and non leaf vertex of $T$ is adjacent from a support vertex of $T$. Let $x$ be a vertex which is not a support vertex nor a leaf. We make use now of the process presented previously to Algorithm \ref{algor} and the notation used there (specially the constructed set $B$). We first note that there must be an intermediate tree $T_i$, in which $x$ is a leaf of $T_i$ adjacent from a support vertex $x'$, which has at least one out-neighbor as a leaf (a leaf of $T$) other than $x$. Thus, in the process of adding vertices to $B$, the existence of such vertex $x$ does not influence on the number of vertices added to $B$, in correspondence with support vertices. In such case, we observe that $|B|=s$, and so, $\rho(T)=|B|=s$.

On the contrary, assume $\rho(T)=s$. Let $n\ne \ell+s$ and $x$ be a vertex which is neither a support vertex nor a leaf. Suppose that $x$ is not adjacent from a support vertex. So, the subset $B$ containing exactly one leaf for each support vertex along with the vertex $x$ is a packing set in $T$. So, $\rho(T)\geq|B|=s+1$, a contradiction. This completes the proof.

(ii) The equality, $\gamma(T)=\lceil(n-\ell+s)/2\rceil$ means equality in all inequalities given in the chain (\ref{eq54}). Thus, since $\gamma(T')=\lceil(n(T'))/2\rceil$ if and only if $T'\in \Phi$ (by using the results of \cite{ha}), we immediately obtain our result.
\end{proof}

The upper bounds in the next theorem were given by Lee (\cite{l1},\cite{l2}).  Note that the upper bound in Theorem \ref{T1} improves both of them.

\begin{theorem}
The following statements hold.
\begin{itemize}
  \item[{\rm(i)}] \emph{(\cite{l1})}  For any rooted tree of order $n$, $\gamma(T)\leq\lceil n/2\rceil$.
  \item[{\rm(ii)}] \emph{(\cite{l2})}  For any binary tree of order $n$, $\gamma(T)\leq(n-1)/2$.
\end{itemize}
\end{theorem}

%%%%%%%%%%%%%%%%%%%%%%%%%%%%%%%%%%%%%%%%%%%%%%%%%%%%%%%%%%%%%%%%%%%%%%%%%%%%%%%%%%%%%%%%%%%%%

\section{General and connected contrafunctional digraphs}\label{sect-gen-contrafunc}

\ \ \ \ In this section, we first bound the packing number $\rho(D)$ of a general digraph $D$ from below. Note that $\delta=\delta(D)$ and $\Delta=\Delta(D)$ are the minimum and maximum degree among the vertices of the underlying graph of $D$, respectively.

\begin{theorem}
Let $D$ be a digraph of order $n$ and let $\delta^{*}=\delta^{*}(D)$ be the minimum in-degree taken over all vertices of minimum degree. Then,
$$\rho(D)\geq \frac{n+\Delta-\delta+(\Delta^{+}-1)(\Delta^{-}-\delta^{*})}{1+\Delta+\Delta^{-}(\Delta^{+}-1)},$$
and this bound is sharp.
\end{theorem}

\begin{proof}
Let $G$ be the underlying graph of $D$. It suffices to construct a packing set $B$ of order at least the lower bound. We construct such subset $B\subseteq V(D)$ as follows. Let $u$ be vertex of minimum degree $\delta$ in $G$, for which $\mbox{deg}^{-}(u)=\delta^{*}$ in $D$. We consider $u$ as a member of $B$ and define $D_{u}=D-(N_G[u]\cup(\cup_{v\in N_D^-(u)}N_D^{+}(v)))$. We iterate this process for the remaining digraph $D_{u}$ until it is empty. It is easy to see that $B$ is a packing set after the final step. In the first step, we removed at most $1+\delta+\delta^{*}(\Delta^{+}-1)$ vertices, and in each of the following steps we removed at most $1+\Delta+\Delta^{-}(\Delta^{+}-1)$ vertices. This yields the following inequality:
$$1+\delta+\delta^{*}(\Delta^{+}-1)+(|B|-1)(1+\Delta+\Delta^{-}(\Delta^{+}-1))\geq n.$$
This implies the lower bound. To see the bound is sharp, we consider the directed star $S_{n}$.
\end{proof}

For the remaining part of this section we investigate the relationship between domination number and packing number of a connected contrafunctional digraph. Harary et al. \cite{hnc} characterized all contrafunctional digraphs as follows.

\begin{lemma}\emph{(\cite{hnc})}\label{D}
The following statements are equivalent for a connected digraph $D$.
\begin{itemize}
  \item[{\rm(i)}] $D$ is contrafunctional.
  \item[{\rm(ii)}] $D$ has exactly one directed cycle $C$ and the removal of any arc $(u,v)$ of $C$ results in a rooted tree with root $v$.
\end{itemize}
\end{lemma}

In fact, every connected contrafunctional digraph can be obtained by adding an arc $(v,r)$ to a rooted tree with root $r$, and every connected contrafunctional digraph gives a rooted tree by eliminating an arbitrary arc of its unique directed cycle. Hao \cite{ha}, defined the height $h(D)$ of a connected contrafunctional digraph $D$ as $\max\{d_{D}(v,V(C))\mid v\in V(D)\}$. Similarly to the discussion for a rooted tree, we define a special sequence corresponding to the connected contrafunctional digraph $D$ with $h(D)\geq2$, as follows. Let $D_{1}=D$. We select a leaf $v_{1}$ with maximum distance from $C$ and let $u_{1}$ be its in-neighbor. Let $D_{2}=D_{1}-N_{+}[u_{1}]$. Iterate this process for the remaining connected contrafunctional digraph $D_{i}$ until $D_{p}$ is the directed cycle $C$ or a connected contrafunctional digraph with height one. We denoted such sequence by $D_{1},\dots,D_{p-1},D_{p}$ (where $D_1=D$) and call it a RDSES of $D$.

In what follows, we give the exact value of $\gamma(D)$ in terms of the packing number $\rho(D)$, for each contrafunctional digraph $D$. To this end, we first present the following necessary lemma.

\begin{lemma}\label{E}
For any contrafunctional digraph $D$ with $h(D)=1$, $\gamma(D)=\rho(D)$.
\end{lemma}

\begin{proof}
Let $C$ be the unique directed cycle of $D$. If every vertex on $C$ is a support vertex, then it is easy to see that $\gamma(D)=\rho(D)=|V(C)|$. Thus, we may assume that some vertices on $C$ are not support vertices. We choose $v\in V(C)$ which is not a support vertex such that its in-neighbor, say $u$, is a support vertex. Consider $D-(u,v)$ and the RDSES $T_{1},\dots,T_{p-1},T_{p}$ of it. If $T_{p}=\emptyset$, then the maximum packing $B=\{v_{1},\dots,v_{p-1}\}$ of $D-(u,v)$ is a packing in $D$. So, $\rho(D)\geq|B|=\rho(D-(u,v))=\rho(D-(u,v))\geq \gamma(D)$. Therefore, $\gamma(D)=\rho(D)$. Now let $T_{p}$ be the isolated vertex $v$. Then, $S=\{u_{1},\dots,u_{p-1}\}$ is a dominating set in $D$. On the other hand, $B$ is a packing in $D$. Therefore, $p-1\geq \gamma(D)\geq \rho(D)\geq|B|=p-1$. Hence, $\gamma(D)=\rho(D)=p-1$, which completes the proof.
\end{proof}

We define $\Omega$ as the family of all connected contrafunctional digraphs $D$ which has a RDSES $D_{1},\dots,D_{p-1},D_{p}$ in which $D_{p}$ is an odd directed cycle. We are now in a position to present the main theorem of this section.

\begin{theorem}
For any connected contrafunctional digraph $D$,
\begin{equation}
\gamma‎‎(D)=‎\left \{
\begin{array}{lll}
‎\rho(D)+1,& \mbox{if} & D\in \Omega,‎ ‎‎‎‎ \vspace{1.5mm}\\
\rho(D),‎ & \mbox{if} & D\notin \Omega.
\end{array}
\right.
\end{equation}
\end{theorem}

\begin{proof}
We have $\rho(D)\leq \gamma(D)$, by Remark \ref{EQ2}. Now let $B$ be a maximum packing in $D-(u,v)$, in which $(u,v)$ is an arc on the unique cycle of $D$. Then $\rho(D-(u,v))=\gamma(D-(u,v))$, by Theorem \ref{T2}. It is easy to verify that $B\setminus\{v\}$ is a packing in $D$. Thus,
$$\rho(D-(u,v))-1\leq|B\setminus\{v\}|\leq \rho(D).$$
Therefore, $\gamma(D)\leq \gamma(D-(u,v))\leq \rho(D)+1$ and consequently, $\gamma(D)=\rho(D)$ or $\gamma(D)=\rho(D)+1$.

We consider a RDSES $D_{1},\dots,D_{p-1},D_{p}$ of $D$. Let $B$ be a maximum packing of $D_{p}$. Clearly, $B\cup\{v_{1},\dots,v_{p-1}\}$ is a maximum packing of $D$. Therefore, $\rho(D)=\rho(D_{p})+p-1$. Similarly, $\gamma(D)=\gamma(D_{p})+p-1$. So, $\gamma(D)=\rho(D)+1$ if and only if $\gamma(D_{p})=\rho(D_{p})+1$. By Lemma \ref{E}, and since the packing and domination numbers of an even directed cycle are the same, we have $\gamma(D)=\rho(D)+1$ if and only if $D_{p}$ is an odd directed cycle, and the proof is completed.
\end{proof}

%%%%%%%%%%%%%%%%%%%%%%%%%%%%%%%%%%%%%%%%%%%%%%%%%%%%%%%%%%%%%%%%%%%%%%%

\section{Total and open domination}\label{sect-total-open}

\ \ \ \ In this section, we consider the total and the open domination numbers in digraphs. Clearly, these parameters exist for a digraph $D$ if and only if $D$ has no isolated vertices (or equivalently $\delta^{-}(D)\geq1$). So, whenever this parameter appears, we assume that this condition is satisfied.

Arumugam et al. \cite{ajv} proved that $2n/(2\Delta^++1)$ is a lower bound on $\gamma_{t}(D)$, for any digraph $D$ of order $n$ without isolated vertices. They also bounded $\gamma_{o}(D)$ from below by $n/\Delta^+$, for any digraph $D$ of order $n$ with $\delta^{-}(D)\geq1$. In consequence, they raised up the following problems.

\begin{p}\label{p1}
Characterize the class of digraphs $D$ for which $\gamma_{t}(D)=2n/(2\Delta^++1)$.
\end{p}
\begin{p}\label{p2}
Characterize the class of digraphs $D$ for which $\gamma_{o}(D)=n/\Delta^+$.
\end{p}

At next we solve these problems. The solution to the second problem is along the similar lines to the first one but different in structures. For the sake of completeness we describe it, as well.

To solve the first problem, we construct a family $\Theta$ of digraphs $D$ as follows. Let $D'$ be a digraph with vertex set $V(D')=\{u_{1},v_{1},\dots,u_{r},v_{r}\}$ and arc set $A(D')=\{(u_{1},v_{1}),\dots,(u_{r},v_{r})\}$. Add $k$ private out-neighbors with respect to $V(D')$ for $u_{1},\dots, u_{r}$, and $k+1$ private out-neighbors with respect to $V(D')$ for $v_{1},\dots, v_{r}$. Let
$$V(D)=V(D')\cup\left(\bigcup_{i=1}^rpn^{+}(u_{i},V(D'))\right)\cup\left(\bigcup_{i=1}^rpn^{+}(v_{i},V(D'))\right).$$
We add some arcs among the vertices in $V(D)\setminus V(D')$ and some arcs $(v,u_{i})$ and $(v,v_{j})$, for some $v\in V(D)\setminus V(D')$ and $1\leq i,j\leq r$, such that $deg^+(v)\leq k+1$, for all $v\in V(D)\setminus V(D')$. Clearly, every vertex in $V(D)\setminus V(D')$ is adjacent from exactly one vertex in $V(D')$. Moreover, $\Delta^+(D)=k+1$.

To solve the second problem, we construct the family $\Sigma$ of digraphs $D$ as follows. Let $D'$ be a contrafunctional digraph and $k\geq \Delta^{+}(D')$. We add $k-deg_{D'}^{+}(v)$ private out-neighbors with respect to $V(D')$ for each vertex $v$ of $D'$. Let
$$V(D)=V(D')\cup\left(\bigcup_{v\in V(D')}pn^{+}(v,V(D'))\right).$$
Add some arcs among the vertices in $V(D)\setminus V(D')$ and some arcs $(u,v)$, for some $u\in V(D)\setminus V(D')$ and $v\in V(D')$, such that $deg^+(u)\leq k$, for all $u\in V(D)\setminus V(D')$. Clearly, every vertex in $V(D)$ is adjacent from exactly one vertex in $V(D')$. Moreover, $\Delta^+(D)=k$.

We are now in a position to present the following theorem.
\begin{theorem}\label{T4}
Let $D$ be a digraph of order $n$ and maximum out-degree $\Delta^+$. Then, the following statements hold.
\begin{enumerate}[{\rm (i)}]
  \item $\gamma_{t}(D)=2n/(2\Delta^++1)$ if and only if $D\in \Theta$.
  \item $\gamma_{o}(D)=n/\Delta^+$ if and only if $D\in \Sigma$.
\end{enumerate}
\end{theorem}

\begin{proof}
(i) We need to restate the proof of the lower bound in order to prove our result. Let $Q$ be a $\gamma_t(D)$-set. Every vertex in $Q$ has at most $\Delta^+$ out-neighbors, and each vertex in $V(D)\setminus Q$ has at least one in-neighbor in $Q$, by the definition. Furthermore, $|(Q,Q)_{D}|\geq\lceil|Q|/2\rceil$. Therefore,
\begin{equation}\label{EQ3}
\Delta^+|Q|\geq|(Q,V(D)\setminus Q)_{D}|+|(Q,Q)_{D}|\geq n-\lfloor|Q|/2\rfloor\geq n-|Q|/2.
\end{equation}
Thus, $\gamma_{t}(D)\geq2n/(2\Delta^++1)$.

Let $D\in \Theta$ and let $S'=V(D')$. Clearly, $S'$ is a total dominating set in $D$. Since every vertex in $V(D)\setminus S'$ is adjacent from exactly one vertex in $S'$, we have $|(S',V(D)\setminus S')_{D}|=n-|S'|$. Moreover, $D\langle S'\rangle$ is the disjoint union of directed paths $P_{2}$ and therefore $|(S',S')_{D}|=|S'|/2$. On the other hand, $\Delta^+|S'|=|(S',V(D)\setminus S')_{D}|+|(S',S')_{D}|$. Thus $\Delta^+|S'|=n-|S'|/2$. Therefore, $\gamma_{t}(D)\leq|S'|=2n/(2\Delta^++1)$, which implies the equality.

Conversely, suppose that the equality holds and let $S$ be a $\gamma_t(D)$-set. Thus, all the inequalities in (\ref{EQ3}) must be equalities (whether $S$ is used instead of $Q$). Now, since $|(S,S)_{D}|\ge |S|/2$, it must happen (by using (\ref{EQ3}) again) that $\Delta^+|S|=|(S,V(D)\setminus S)_{D}|+|(S,S)_{D}|$, $|(S,V(D)\setminus S)_{D}|=n-|S|$ and $|(S,S)_{D}|=|S|/2$. This shows that every vertex in $V(D)\setminus S$ is adjacent from exactly one vertex in $S$ and also, since $D\langle S\rangle$ has no isolated vertices, that $D\langle S\rangle$ is a disjoint union of paths $P_{2}$ with arcs $(u'_{1},v'_{1})$,$\dots$,$(u'_{|S|/2},v'_{|S|/2})$ for some set of vertices $\{u'_{1},v'_{1},\dots,u'_{|S|/2},v'_{|S|/2}\}$. Since $\Delta^+|S|=|(S,V(D)\setminus S)_{D}|+|(S,S)_{D}|$, each vertex in $S$ has out-degree $\Delta^{+}$. This implies that each $u'_{i}$ has $\Delta^+-1$ (private) out-neighbors in $V(D)\setminus S$ and that each $v'_{i}$ has $\Delta^+$ (private) out-neighbors in $V(D)\setminus S$, for all $1\leq i\leq|S|/2$. Therefore, we can easily deduce that $D\in \Theta$.

(ii) We first present a proof for the inequality $\gamma_{o}(D)\geq n/\Delta^{+}$. Let $Q$ be a $\gamma_o(D)$-set. Since every vertex in $V(D)$ is adjacent from at least one vertex in $Q$, we have
\begin{equation}\label{EQ4}
|(Q,Q)_{D}|=\sum_{v\in Q}deg_{Q}^{+}(v)=\sum_{v\in Q}deg_{Q}^{-}(v)\geq|Q|
\end{equation}
and
\begin{equation}\label{EQ5}
\Delta^+|Q|\geq|(Q,V(D)\setminus Q)_{D}|+|(Q,Q)_{D}|\geq n-|Q|+|Q|.
\end{equation}
Thus, $\gamma_{o}(D)=|Q|\geq n/\Delta^+$.

Let $D\in \Sigma$ and let $S'=V(D')$. It is easy to see that $S'$ is an open dominating set in $D$ and that $|(S',S')_{D}|=\sum_{v\in S'}deg_{S'}^{-}(v)=\sum_{v\in S'}deg_{S'}^{+}(v)=|S'|$. Furthermore, every vertex in $V(D)\setminus S'$ has exactly one in-neighbor in $S'$. On the other hand,
$$|(S',V(D)\setminus S')_{D}|=\sum_{v\in S'}deg_{V(D)\setminus S'}^{+}(v)=\Delta^{+}|S'|-\sum_{v\in S'}deg_{S'}^{+}(v)=|S'|(\Delta^{+}-1).$$
So, the inequalities in (\ref{EQ4}) and (\ref{EQ5}) hold with equality when we replace $Q$ by $S'$. Therefore, $\gamma_{o}(D)\leq|S'|=n/\Delta^+$. This implies the equality.

Suppose now that the equality holds and that $S$ is a $\gamma_{o}(D)$-set in $D$. Then, the inequalities in (\ref{EQ4}) and (\ref{EQ5}) hold with equality when we use $S$ instead of $Q$. Since $|(S,S)_{D}|=|S|$ and $S$ is an open dominating set, $deg_{S}^{-}(v)=1$ for all vertices $v\in S$. Therefore, $D\langle S\rangle$ is a contrafunctional digraph. Suppose now that there exists a vertex $v\in S$ for which $deg_{V(D)\setminus S}^{+}(v)<\Delta^{+}-deg_{S}^{+}(v)$. Then, $|(S,V(D)\setminus S)_{D}|<\Delta^{+}|S|-\sum_{v\in S}deg_{S}^{+}(v)=|S|(\Delta^{+}-1)$, a contradiction. Therefore, $deg_{V(D)\setminus S}^{+}(v)=\Delta^{+}-deg_{S}^{+}(v)$ for each vertex $v\in S$. On the other hand, $|(S,V(D)\setminus S)_{D}|=n-|S|$ shows that every vertex in $V(D)\setminus S$ has exactly one in-neighbor in $S$. So, every vertex $v\in S$ has $\Delta^{+}-deg_{S}^{+}(v)$ private out-neighbors in $V(D)\setminus S$ and therefore, $D\in \Sigma$.
\end{proof}

Hao and Chen \cite{hc} introduced the {\em out-Slater number} $s\ell^{+}(D)$ of a digraph $D$ of order $n$ as min$\{k\mid\lfloor k/2\rfloor+d^{+}_{1}+\cdots+d^{+}_{k}\geq n\}$, where $d^{+}_{1},\cdots,d^{+}_{k}$ are the first $k$ largest out-degrees of $D$. Among other results, they showed that
\begin{equation}\label{EQU1}
\gamma_{t}(D)\geq s\ell^{+}(D),
\end{equation}
for all digraphs $D$ with no isolated vertices. Also, for a directed tree $T$ of order $n\geq2$ with $\ell$ leaves, they proved that
\begin{equation}\label{EQU2}
s\ell^{+}(T)\geq2(n-\ell+1)/3.
\end{equation}
%For more information on the out-Slater number the reader can consult \cite{hc}.

From now on we bound $\gamma_{t}(T)$ from above for a directed tree $T$ of order $n\geq2$ just in terms of the out Slater number.
\begin{theorem}
Let $T$ be a directed tree of order $n\geq2$. Then,
$$s\ell^{+}(T)\leq \gamma_{t}(T)\leq \frac{3}{2}s\ell^{+}(T)-1.$$
Moreover, all integer values between the lower and upper bounds are realizable.
\end{theorem}

\begin{proof}
The lower bound is that from \cite{hc}. On the other hand, the upper bound holds for the directed star $S_{n}$ with $s\ell^{+}(S_{n})=\gamma_{t}(S_{n})=2$. So, we may assume that $T$ is different from $S_{n}$. Clearly, the set of all non-leaf vertices of $T$ is a total dominating set in $D$ and so,
\begin{equation}\label{EQU3}
\gamma_{t}(T)\leq n-\ell.
\end{equation}
Now the desired upper bound follows from (\ref{EQU2}) and (\ref{EQU3}). To show that all values between the lower and upper bounds are realizable, it suffices to prove that for any integer $a\geq 2$ and $0\leq b\leq\lfloor a/2\rfloor-1$, there exists a rooted tree $T$ such that $s\ell^{+}(T)=a$ and $\gamma_{t}(T)=a+b$.

Consider $T'$ as a rooted tree obtained from a directed path $P_{a}$ on $a\geq2$ vertices $v_{1},\cdots,v_{a}$, consecutively, by adding $2a$ leaves adjacent from each vertex of $P_{a}$. Let $T$ be the rooted tree obtained from $T'$ by replacing exactly one pendant arc $(v_{i},v)$ with a directed path $v_{i},w_{i},v$ through a new vertex $w_{i}$, for every $0\leq i\leq b$. It is easy to check that $n=2a^{2}+a+b$ and that $\gamma_{t}(T)=a+b$. If $d^{+}_{1}\geq \cdots \geq d^{+}_{n}$ is the non-increasing out-degree sequence of $T$, then
$$\lfloor a/2\rfloor+d^{+}_{1}+\cdots+d^{+}_{a}=\lfloor a/2\rfloor+deg^{+}(v_{1})+\cdots+deg^{+}(v_{a})=2a^{2}+a+\lfloor a/2\rfloor-1\geq n.$$
Thus $s\ell^{+}(T)\leq a$. Suppose to the contrary that $s\ell^{+}(T)=k\leq a-1$. Then,
\begin{equation*}
\begin{array}{lcl}
\lfloor k/2\rfloor+d^{+}_{1}+\cdots+d^{+}_{k}\leq \lfloor a/2\rfloor+d^{+}_{1}+\cdots+d^{+}_{a-1}&=&\lfloor a/2\rfloor+deg^{+}(v_{1})+\cdots+deg^{+}(v_{a-1})\\
&=&2a^{2}-\lceil a/2\rceil-1<n,
\end{array}
\end{equation*}
which is a contradiction. Therefore, $s\ell^{+}(T)=a$ and this completes the proof.
\end{proof}

\section{Concluding remarks and open problems}

\ \ \ We have studied several relationships between the packing number and the open, total and standard domination numbers of digraphs. We have dedicated special attention to the directed trees and contrafunctional digraphs while proving our results. As a remarkable aspect, we have settled two problems presented in [Australas. J. Combin. 39 (2007), 283--292]. Finally, as future research activities we next point out two questions that we consider would be interesting to be dealt with.
\begin{itemize}
  \item We first notice that any total dominating set in a digraph $D$ is also a dominating set of $D$. Moreover, if we consider a dominating set $S$ of $D$, by taking the set $S$ and one neighbor (not in $S$) of each vertex of $S$ we can easily construct a total dominating set of $D$. Thus, we can clearly deduce the following bounds. For any digraph $D$,
$$\gamma(D)\le \gamma_t(D)\le 2\gamma(D).$$ An equivalent result is well known for graphs, and the problem of characterizing the equality in these bounds remains open for graphs (see the survey \cite{Hen}). Thus, it is worthwhile to consider the equivalent problem for digraphs.
  \item There is not much knowledge about complexity aspects of domination parameters in digraphs. According to this fact, we think will deserve the attention to study the computational complexity of computing the packing number of digraphs.
\end{itemize}

%\begin{flushleft}\textbf{Acknowledgments}\end{flushleft}\vspace{-1.5mm}
%\ \ We thank S.M. Hosseini Moghaddam for providing us with the structure $G_{D}$ in Remark \ref{SM} and giving the first part of Lemma \ref{A}.

%%%%%%%%%%%%%%%%%%%%%%%%%%%%%%%%%%%%%%%%%%%%%%%%%%%%%%%%%%%%%%%%%%%%%%%%%%%

\end{document}